\documentclass[10pt]{amsart}
\usepackage{amssymb}
\usepackage{epsfig}
\usepackage{url}
\usepackage{setspace}
\usepackage{color}
\theoremstyle{plain}

\newtheorem{thm}{Theorem}[section]
\newtheorem{cor}[thm]{Corollary}
\newtheorem{lem}[thm]{Lemma}
\newtheorem{prop}[thm]{Proposition}
\newtheorem{rem}[thm]{Remark}
\newtheorem{ques}[thm]{Question}
\newtheorem{conj}[thm]{Conjecture}
\newtheorem{prob}[thm]{Problem}

\def\bbb{\mathbb}
\def\op{\operatorname}
\renewcommand{\phi}{\varphi}
\newcommand{\R}{\bbb{R}}
\newcommand{\N}{\bbb{N}}
\newcommand{\Z}{\bbb{Z}}
\newcommand{\Q}{\bbb{Q}}

\newcommand{\C}{\bbb{C}}

\makeatletter
\let\@@pmod\pmod
\DeclareRobustCommand{\pmod}{\@ifstar\@pmods\@@pmod}
\def\@pmods#1{\mkern4mu({\operator@font mod}\mkern 6mu#1)}
\makeatother

\begin{document}
\title[Signs behaviour of sums of weighted numbers of compositions]{Signs behaviour of sums of weighted numbers of compositions}
\author{Filip Gawron and Maciej Ulas}

\keywords{composition; sum; sign; sums of compositions} \subjclass[2010]{}
\thanks{Research of the authors was supported by a grant of the National Science Centre (NCN), Poland, no. UMO-2019/34/E/ST1/00094}

\begin{abstract}
Let $A$ be a subset of positive integers. For a given positive integer $n$ and $0\leq i\leq n$ let $c_{A}(i,n)$ denotes the number of $A$-compositions of $n$ with exactly $i$ parts. In this note we investigate the sign behaviour of the sequence $(S_{A,k}(n))_{n\in\N}$, where $S_{A,k}(n)=\sum_{i=0}^{n}(-1)^{k}i^{k}c_{A}(i,n)$. We prove that for a broad class of subsets $A$, the number $(-1)^{n}S_{A,k}(n)$ is non-negative for all sufficiently large $n$. Moreover, we show that there is $A\subset \N_{+}$ such that the sign behaviour of $S_{A,k}(n)$ is not periodic.
\end{abstract}

\maketitle

\section{Introduction}\label{sec1}

Let $\N$ be the set of non-negative integers and $\N_{+}$ denotes the set of positive integers. Let $A\subset \N_{+}$ be given. With the set $A$ and a non-negative integer $n$, one can connect two combinatorial objects: an $A$-partition of $n$ and an $A$-composition of $n$. By $A$-partition of $n$, we understand a finite non-increasing sequence $(\lambda_{1},\ldots, \lambda_{k})$, where $\lambda_{i}\in A$ and $n=\sum_{i=1}^{k}\lambda_{i}$.
By $A$-composition of $n$, we understand a finite sequence $(\lambda_{1},\ldots, \lambda_{k})$, where $\lambda_{i}\in A$ and $n=\sum_{i=1}^{k}\lambda_{i}$. The difference is clear. In $A$-compositions the order of summands in $n=\sum_{i=1}^{k}\lambda_{i}$ gives different $A$-compositions. To see an example, let us consider $A=\{1, 2, 3\}$. We have four $A$-partitions of 4 given by
$$
(3, 1), (2, 2), (2, 1, 1), (1, 1, 1, 1).
$$
On the other hand, we have seven $A$-compositions of 4 given by
$$
(3, 1), (1, 3), (2, 2), (2, 1, 1), (1, 2, 1), (1, 1, 2) , (1, 1, 1, 1).
$$
In the sequel, by $p_{A}(n)$ and $c_{A}(n)$ we denote the number of $A$-partitions and $A$-compositions of $n$, respectively. It is clear that $c_{A}(n)>p_{A}(n)$ for each $A\subset\N_{+}$, where $\#A\geq 2$ and $n>\op{min}(A)$. From the general theory we know that
$$
P_{A}(x)=\prod_{a\in A}\frac{1}{1-x^{a}}=\sum_{n=0}^{\infty}p_{A}(n)x^{n}
$$
is an ordinary generating function for the sequence $(p_{A}(n))_{n\in\N}$. One can also consider a polynomial $p_{A, n}(t)$, called $n$-th $A$-partition polynomial, which can be seen as a polynomial analogue of the number $p_{A}(n)$. More precisely, we have
$$
P_{A}(t,x)=\prod_{a\in A}\frac{1}{1-tx^{a}}=\sum_{n=0}^{\infty}p_{A, n}(t)x^{n}.
$$
We write $p_{A,n}(t)=\sum_{i=0}^{n}p_{A}(i, n)t^{n}$ and observe that $p_{A}(i, n)$ is the number of $A$-partitions of $n$ with exactly $i$ parts. In
particular, $p_{A,n}(1)=\sum_{i=0}^{n}p_{A}(i, n)$ is the number of $A$-partitions of $n$. In a recent paper \cite{GU}, we investigated the sign behaviour of the sequence $\left(\sum_{i=0}^{n}(-1)^{i}i^{k}p_{A}(i,n)\right)_{n\in\N}$, where $k\in\N$ is fixed and $n\in\N$. We proved that for a broad class of sets $A$ (including all sets containing only odd numbers) we have $(-1)^{n}\sum_{i=0}^{n}(-1)^{i}i^{k}p_{A}(i,n)\geq 0$ for all $k, n\in\N$.

In this paper we consider related problem and instead of $A$-partitions we consider $A$-compositions. To formulate our basic problem we need some notation. Let $A\subset \N_{+}$ and consider the ordinary generating function $f_{A}(x)$ of the characteristic function of the set $A$, i.e., we write
$$
f_{A}(x)=\sum_{a\in A}x^{a}.
$$
Then
$$
C_{A}(x)=\frac{1}{1-f_{A}(x)}=\sum_{n=0}^{\infty}c_{A}(n)x^{n}
$$
is the ordinary generating function of the sequence $(c_{A}(n))_{n\in\N}$.

Then, the sequence of coefficients $(C_{A,n}(t))_{n\in\N}$ in the power series expansion
$$
C_{A}(t,x)=\frac{1}{1-tf_{A}(x)}=\sum_{n=0}^{\infty}f_{A}(x)^{n}t^{n}=\sum_{n=0}^{\infty}f_{A,n}(t)x^{n}
$$
have a natural combinatorial interpretation, similar to the interpretation of the coefficients of the $n$-th $A$-partition polynomial $p_{A, n}(t)$. More precisely, let us write
$$
f_{A,n}(t)=\sum_{i=0}^{n}c_{A}(i,n)t^{n}.
$$
Then, the $i$th coefficient $c_{A}(i, n)$ of $f_{A,n}(t)$ is just the number of compositions of an integer $n$ into exactly $i$ elements of the set $A$. In particular, we have that  $f_{A,n}(1)=\sum_{i=0}^{n}c_{A}(i,n)=c_{A}(n)$ - the number of an $A$-compositions of $n$. Similarly as in the case of $A$-partitions we investigate the family of sums
$$
S_{A, k}(n)=\sum_{i=0}^{n}(-1)^{i}i^{k}c_{A}(i,n),
$$
where $k, n\in\N$, and formulate the following general question.

\begin{ques}
Let $A\subset\N_{+}$ and $k\in\N$ be fixed. What can be said about the behaviour of the sequence $(\op{sign}(S_{A, k}(n)))_{n\in\N}$? In particular, under which conditions on $A$, we have that $(-1)^{n}S_{A, k}(n)\geq 0$ for a given $k$ and all but finitely many $n$?
\end{ques}

As we will see, there are some similarities in behaviour of signs of the sequence $(S_{A, k}(n))_{n\in\N}$ in case of $A$-partitions and $A$-compositions as well as striking differences between them.

Let us describe the content of the paper in some details. In Section \ref{sec2} we collect basic results concerning the sequence of $A$-partition polynomials and deduce certain recurrence relations for the double sequence $(S_{A, k}(n))_{k, n\in\N}$. In particular, we obtain convolution-like identity connecting $S_{A, k+1}(n)$ with $S_{A, j}(n-i)S_{A, k-j}(i)$ for $j\in\{0, \ldots, k\}$ and $i\in\{0, \ldots, n-1\}$. In Section \ref{sec3} we prove that if $A$ contains only odd numbers, then $(-1)^{n}S_{A, k}(n)\geq 0$ for all $k, n\in\N$. Moreover, the same statement is true for sets of the form $A=\N_{+}\setminus E$, where $E\subset\N_{+}$ contains only even integers. Finally, in the last section, we formulate some problems and questions which may stimulate further research.

\section{Basic observations}\label{sec2}

In order to investigate the sign behaviour of the sequence $(S_{A,k}(n))_{n\in\N}$, where $k\in\N$ is fixed, we consider the differential operator $\delta:\;\R[t]\rightarrow \R[t]$ defined in the following way: for $f\in\R[t]$ we put
$$
\delta(f(t))=t\frac{df(t)}{dt}.
$$
For $k=0$ we put $\delta^{(0)}(f)=f$, and for $k\in\N_{+}$ we define by induction a $k$-th power of the operator $\delta$ as $\delta^{(k)}(f)$, where
$$
\delta^{(k)}(f)=\delta(\delta^{(k-1)}(f)).
$$
The basic properties of the operator $\delta$ can be summarized in the following well known result (which can be also found in \cite[Lemma 2.1]{GU}).

\begin{lem}\label{delta}
\begin{enumerate}
\item $\delta$ is linear operator, i.e., for all $u, v\in\R, f, g\in\R[t]$ we have $\delta(uf+vg)=u\delta(f)+v\delta(g)$;
\item if $f\in\R[t]\setminus\R$ then $\op{deg}\delta(f)=\op{deg}f$;
\item for $k\in\N$ we have
$$
\delta^{(k)}(fg)=\sum_{i=0}^{k}\binom{k}{i}\delta^{(i)}(f)\delta^{(k-i)}(g).
$$
\item if $f(t)=\sum_{i=0}^{n}a_{i}t^{i}\in\R[t]$ and $k\in\N$ is given then
$$
\delta^{(n)}(f(t))=\sum_{i=0}^{n}a_{i}i^{n}t^{i}
$$
\end{enumerate}
\end{lem}
\begin{proof}
The first two properties are easy consequences of the definition of the operator $\delta$. The proof of the third property is just the application of the Leibnitz rule together with induction on $k$. The last property follows from the linearity of the operator $\delta$ and the equality $\delta^{(k)}(t^{i})=i^{k}t^{i}$.
\end{proof}

As an immediate consequence of the fourth property of the operator $\delta$ we get the equality
$$
S_{A,k}(n)=\sum_{i=0}^{n}(-1)^{i}i^{k}c_{A}(i,n)=\delta^{(k)}(f_{A,n}(t))\mid_{t=-1}.
$$

For a given set $V\in\N_{+}$ we put
$$
V(n):=V\cap\{1, 2, \ldots, n\},
$$
i.e., $V(n)$ contains those elements of $V$ which are $\leq n$.

We start with the following general result.

\begin{thm}\label{genthm}
Let $A\in \N_{+}$ and consider the sequence $(f_{A,n}(t))_{n\in\N}$.
\begin{enumerate}
 \item We have $f_{A,0}(t)=1$ and for $n\in\N_{+}$ the following recurrence relation holds:
 $$
 f_{A,n}(t)=t\sum_{a\in A(n)}f_{A,n-a}(t).
 $$
 \item We have the following identity
 $$
 \sum_{a\in A(n)}a\delta(f_{A,n-a})(t)=\sum_{a\in A(n)}(n-a)f_{A,n-a}(t).
 $$
 \item We have the following identity
 $$
 f_{A,n}(t)+f_{A,n}(-t)=2\sum_{i=0}^{n}f_{A,i}(-t)f_{A,n-i}(t).
 $$
 \item
 Let $A=O_{A}\cup E_{A}$, where $O_{A}=\{a\in A:\;a\equiv 1\pmod*{2}\}$ and $E_{A}=A\setminus O_{A}=\{a\in A:\;a\equiv 0\pmod*{2}\}$. Then, we have the following identity
 $$
\frac{1}{2}\sum_{i=0}^{n}f_{O_{A},i}(-t)\left(f_{A,n-i}(-t)+(-1)^{n-i}f_{A,n-i}(t)\right)=\sum_{i=0}^{n}(-1)^{i}f_{A,i}(t)f_{A,n-i}(-t).
$$
In particular, if $E_{A}=\emptyset$, then for $n\in\N$ we have $f_{A,n}(-t)=(-1)^{n}f_{A,n}(t)$.
\end{enumerate}
\end{thm}
\begin{proof}
 The first part of our theorem is an immediate consequence of the functional relation $(1-tf_{A}(x))C_{A}(t,x)=1$. Indeed, for $n\in\N_{+}$, the $n$-th term in the power series expansion of the left hand side is exactly the difference of $C_{A,n}(t)$ and $t\sum_{a\in A(n)}C_{A,n-a}(t)$, and hence the result.

 In order to get the second identity, let us observe the following equalities
 \begin{equation}\label{partialid}
 \frac{\partial C_{A}(t,x)}{\partial t}=\frac{f_{A}(x)}{(1-tf_{A}(x))^{2}},\quad\quad \frac{\partial C_{A}(t,x)}{\partial x}=\frac{tf'_{A}(x)}{(1-tf_{A}(x))^{2}}.
 \end{equation}
 In consequence, we get the identity
 $$
 tf_{A}'(x)\frac{\partial C_{A}(t,x)}{\partial t}=f_{A}(x)\frac{\partial C_{A}(t,x)}{\partial x}.
 $$
 Comparison of the terms near $x^{n}$ in the expansions
 \begin{align*}
 xtf_{A}'(x)\frac{\partial C_{A}(t,x)}{\partial t}&=\left(\sum_{a\in A}ax^{a}\right)\left(\sum_{n=0}^{\infty}tC_{A,n}'(t)x^{n}\right)=\sum_{n=0}^{\infty}\left(t\sum_{a\in A(n)}aC_{a,n-a}'(t)\right)x^{n},\\
 xf_{A}(x)\frac{\partial C_{A}(t,x)}{\partial x}&=\left(\sum_{a\in A}x^{a}\right)\left(\sum_{n=0}^{\infty}(n)C_{A,n}(t)x^{n}\right)\\
                                               &=\sum_{n=0}^{\infty}\left(\sum_{a\in A(n)}(n-a)C_{A,n-a}(t)\right)x^{n},
 \end{align*}
 gives the required identity.

 In order to prove (3) it is enough to note the functional identity
 $$
 C_{A}(t,x)+C_{A}(-t,x)=\frac{2}{1-t^2f_{A}^{2}(x)}=2C_{A}(t,x)C_{A}(-t,x).
 $$
 Comparing the terms on both sides of the expansions of $C_{A}(t,x)+C_{A}(-t,x)$ and $2C_{A}(t,x)C_{A}(-t,x)$ we get the result.

Finally, in order to get the last identity we note the relation $f_{A}(x)=f_{O_{A}}(x)+f_{E_{A}}(x)$ and $f_{A}(-x)=-2f_{O_{A}}(x)+f_{A}(x)$. From the shape of $C_{A}(t,x)$, we easily get the identity
$$
\frac{1}{C_{A}(-t,x)}+\frac{1}{C_{A}(t,-x)}=2(1+tf_{O_{A}}(x))
$$
and thus
$$
\frac{1}{2}\frac{1}{1+tf_{O_{A}}(x)}(C_{A}(-t,x)+C_{A}(t,-x))=C_{A}(-t,x)C_{A}(t,-x).
$$
Noting that $1/(1+tf_{O_{A}}(x))=C_{O_{A}}(-t,x)$ and comparing the terms near $x^{n}$ on both sides of the above identity, we get:
$$
\frac{1}{2}\sum_{i=0}^{n}f_{O_{A},i}(-t)(f_{A,n-i}(-t)+(-1)^{n-i}f_{A,n-i}(t))=\sum_{i=0}^{n}(-1)^{i}f_{A,i}(t)f_{A,n-i}(-t),
$$
which is exactly the identity we were looking for. If $E_{A}=\emptyset$, then $f_{E_{A}}(x)=0$ and $f_{A}(-x)=f_{O_{A}}(-x)=-f_{O_{A}}(x)=-f_{A}(x)$. In consequence $C_{A}(-t,x)=C_{A}(t,-x)$ and we get the result.

\end{proof}

In the second result we collect identities which will be useful in studying signs behaviour of the sequence $S_{A,k}(n)$ for given $k, n\in\N$.

\begin{thm}\label{tworec}
Let $A\subset\N_{+}$ and consider the power series expansion
$$
\frac{f_{A}(x)}{xf'_{A}(x)}=\sum_{n=0}^{\infty}q_{A}(n)x^{n}.
$$

We have the following identities
$$
\delta(f_{A,n})(t)=\sum_{i=0}^{n}iq_{A}(n-i)f_{A,i}(t)
$$
and
$$
\delta(f_{A,n})(t)=\sum_{i=0}^{n-1}f_{A,n-i}(t)f_{A,i}(t).
$$
\end{thm}
\begin{proof}
To get the first equality let us note that from the identities (\ref{partialid}) we get that
\begin{align*}
\sum_{n=0}^{\infty}&\delta(f_{A,n})(t)x^{n}=t\frac{\partial C_{A}(t,x)}{\partial t}=\frac{f_{A}(x)}{xf'_{A}(x)}\cdot x\frac{\partial C_{A}(t,x)}{\partial x}\\
                   &=\left(\sum_{n=0}^{\infty}q_{A}(n)x^{n}\right)\left(\sum_{n=0}^{\infty}nf_{A,n}(t)x^{n}\right)=\sum_{n=0}^{\infty}\left(\sum_{i=0}^{n}iq_{A}(n-i)f_{A,i}(t)\right)x^{n}.
\end{align*}
Comparing the coefficients on penultimate sides of the above chain of equalities we get the result.

To get the second identity we recall that $f_{A,0}(t)=1$ and note the following
\begin{align*}
\sum_{n=0}^{\infty}&\delta(f_{A,n})(t)x^{n}=t\frac{\partial C_{A}(t,x)}{\partial t}=\frac{tf_{A}(x)}{(1-tf_{A}(x))^{2}}\\
                   &=\frac{1}{(1-tf_{A}(x))^{2}}-\frac{1}{1-tf_{A}(x)}=\sum_{n=0}^{\infty}\left(\sum_{i=0}^{n}f_{A,i}(t)f_{A,n-i}(t)-f_{A,n}(t)\right)x^{n}\\
                   &=\sum_{n=0}^{\infty}\left(\sum_{i=0}^{n-1}f_{A,i}(t)f_{A,n-i}(t)\right)x^{n}.
\end{align*}
Again, comparing coefficients on penultimate sides, we get the result.
\end{proof}

Having proved Theorem \ref{tworec} as an immediate consequence we get the following two recurrence relations satisfied by the double sequence $(S_{A,k}(n))_{k,n}$.

\begin{cor}\label{prodrec}
Let $k, n\in\N$. We have the following identities:
$$
S_{A,k+1}(n)=\sum_{i=0}^{n}iq_{A}(n-i)S_{A,k}(i),
$$
and
$$
S_{A,k+1}(n)=\sum_{i=0}^{n-1}\sum_{j=0}^{k}\binom{k}{j}S_{A,j}(n-i)S_{A,k-j}(i).
$$
In particular, if $(-1)^{n}S_{A,0}(n)\geq 0$ then for each $k\in\N$ we have $(-1)^{n}S_{A,k}(n)\geq 0$.
\end{cor}
\begin{proof}
The first identity is a simple induction on $k$. Indeed, we note that the statement is true for $k=0$. Next, recall that $S_{A,k}(n)=\delta^{(k)}f_{A,n}(t)|_{t=-1}$. Applying the operator $\delta^{(k)}$ to the first identity from Theorem \ref{tworec} and taking $t=-1$ we get the result.

To get the second identity we use the same approach together with Lemma \ref{delta} (3) and by taking $t=-1$ we get the result.

Finally, the "in particular" part follows from a simple induction on $k$. The case $k=0$ is our assumption. Suppose that our statement is true for each integer $m<k$, i.e., for each $n$ we have $(-1)^{n}S_{A,m}(n)\geq 0$. Then we have
$$
(-1)^{n}S_{A,k}(n)=\sum_{i=0}^{n-1}\sum_{j=0}^{k-1}\binom{k-1}{j}[(-1)^{n-i}S_{A,j}(n-i)][(-1)^{i}S_{A,k-j}(i)]\geq 0
$$
and our result follows.
\end{proof}

\begin{rem}
{\rm The second recurrence relation given in Corollary \ref{prodrec} will be of great use in the sequel. On the other side, it seems that the first one, although interesting from a theoretical point of view, is of little use. The problem is that we are not able to control the sign behaviour of the sequence $(q_{A}(n))_{n\in\N}$. Indeed, consider the set $A=\{1, 2, 3\}$. We have $f_{A}(x)=x+x^2+x^3$ and
$$
\frac{f_{A}(x)}{xf_{A}'(x)}=\frac{x^2+x+1}{3 x^2+2 x+1}=\sum_{n=0}^{\infty}q_{A}(n)x^{n}.
$$
From Theorem \ref{negthm} we easily get that the sequence $(\op{sign}(q_{A}(n)))_{n\in\N}$ is not periodic. On the other hand, from Proposition \ref{12m} below, applied to $m=3$, we get that $(-1)^{n}S_{A,k}(n)\geq 0$.
}
\end{rem}
\section{Results}\label{sec3}

The aim of this section is to present several classes of sets $A$ having the property $(-1)^{n}S_{A,k}(n)\geq 0$ for each $k, n\in\N$. We start with the following general observation.

\begin{cor}\label{oddset}
Let $A\subset 2\N+1$. Then for each $k, n\in\N$ we have $(-1)^{n}S_{A,k}(n)\geq 0$.
\end{cor}
\begin{proof}
Let $f_{A}(x)=\sum_{a\in A}x^{a}$. Because $A\subset 2\N+1$ we have $f_{A}(-x)=-f_{A}(x)$. Let us recall that $S_{A,0}(n)=f_{A,n}(-1)$. We thus have
$$
\sum_{n=0}^{\infty}(-1)^{n}S_{A,0}(n)x^{n}=\sum_{n=0}^{\infty}S_{A,0}(n)(-x)^{n}=\frac{1}{1+f_{A}(-x)}=\frac{1}{1-f_{A}(x)}=\sum_{n=0}^{\infty}c_{A}(n)x^{n}.
$$
We thus have $(-1)^{n}S_{A,0}(n)=c_{A}(n)\geq 0$ and hence the result for $k=0$. Applying second part of Corollary \ref{prodrec} we get that $(-1)^{n}S_{A,k}(n)\geq 0$ for each $k, n\in\N$.
\end{proof}

\begin{rem}
{\rm It is interesting fact that exactly the same result is true if we consider {\it partitions} with parts in the set $A$, instead of compositions with parts in the set $A$ (see \cite[Lemma 3.3]{GU}).}
\end{rem}

In the next proposition we write $a_{m}$ for the sequence $(a,\ldots, a)$ containing exactly $m$ occurrences of the symbol $a$.

\begin{prop}\label{12m}
Let $m\in\N_{+}$ be fixed and put $A=\{1,\ldots, m\}$.
\begin{enumerate}
\item If $m$ is odd then the sequence $((-1)^{n}S_{A,0}(n))_{n\in\N}$ is periodic of period $m+1$. More precisely,
$$
(\op{sing}((-1)^{n}S_{A,0}(n)))_{n\in\N}=\overline{(1, 1, 0_{m-1})}
$$
and we have $(-1)^{n}S_{A,k}(n)\geq 0$ for each $k, n\in\N$.
\item If $m$ is even then the sequence $((-1)^{n}S_{A,0}(n))_{n\in\N}$ is periodic of period $2(m+1)$. More precisely,
$$
(\op{sign}((-1)^{n}S_{A,0}(n)))_{n\in\N}=\overline{(1, 1, 0_{m-1},-1 , -1, 0_{m-1})}.
$$
\end{enumerate}
\end{prop}
\begin{proof}
We have $f_{A}(x)=\sum_{i=1}^{m}x^{i}=x\frac{1-x^{m}}{1-x}$.

If $m$ is odd, then we have
\begin{align*}
\sum_{n=0}^{\infty}(-1)^{n}S_{A,0}(n)x^{n}&=F(-1,-x)=\frac{1}{1+f_{A}(-x)}=\frac{1+x}{1-x^{m+1}}\\
                                  &=1+\sum_{n=1}^{\infty}(I_{m}(n)+I_{m}(n-1))x^{n},
\end{align*}
where $I_{m}(n)=1$ if $m+1|n$ and 0 otherwise. It is clear that the expression $I_{m}(n)+I_{m}(n-1)$ is non-negative and thus $(-1)^{n}S_{A,0}(n)\geq 0$. Moreover, $(-1)^{n}S_{A,0}(n)=I_{m}(n)+I_{m}(n-1)=1$ if and only if $n\equiv 0, 1\pmod*{m}$. Thus, we get the first part of our statement. Applying Corollary \ref{prodrec} we get that $(-1)^{n}S_{A,k}(n)\geq 0$ for each $k, n\in\N$.

If $m$ is even, then we have
\begin{align*}
\sum_{n=0}^{\infty}&(-1)^{n}S_{A,0}(n)x^{n}=F(-1,-x)=\frac{1}{1+f_{A}(-x)}=\frac{1+x}{1+x^{m+1}}\\
                   &=\frac{(1+x)(1-x^{m+1})}{1-x^{2(m+1)}}=\sum_{n=0}^{\infty}c_{m}(n)x^{n},
\end{align*}
where $c_{m}(n)=J_{m}(n)+J_{m}(n-1)-J_{m}(n-m-1)-J_{m}(n-m-2)$ and $J_{m}(n)=1$ if $2(m+1)|n$ and 0 otherwise. A careful analysis of the expression for $c_{m}(n)$ for $n=2(m+1)p+q, q\in\{0, 1, \ldots, 2(m+1)-1\}, p\in\N$, reveals that
$$
c_{m}(n)=\begin{cases}\begin{array}{lll}
                        1 & \mbox{for} & n\equiv 0, 1\pmod*{2(m+1)}, \\
                        -1 & \mbox{for} & n\equiv m+1, m+2\pmod*{2(m+1)}, \\
                        0 &  & \mbox{otherwise},
                      \end{array}
\end{cases}
$$
which is exactly the form from the second part of our theorem.
\end{proof}

\begin{rem}
{\rm Let $A=\{1,\ldots, m\}$, where $m$ is even. We proved that for $k=0$, the sequence $(\op{sign}((-1)^{n}S_{A,0}(n)))_{n\in\N}$ is periodic with a pure period $2(m+1)$. Our numerical calculations suggest that this behaviour persist also for $k>0$. We formulate the following.

\begin{conj}
Let $A=\{1,\ldots, m\}$, where $m$ is even. The sequence $((-1)^{n}S_{A, k}(n))_{n\in\N}$ is eventually periodic of period length $2(m+1)$. If $k\geq 2(m+1)$, then the period $T_{m}$ has the form $T_{m}=(1_{m+1},(-1)_{m+1})$.
\end{conj}

}
\end{rem}

\bigskip

\begin{rem}
{\rm An infinite version of Proposition \ref{12m}, i.e., the case $m=+\infty$, or equivalently $A=\N_{+}$ is trivial. Indeed, in this case $f_{A}(x)=x/(1-x)$ and 
$$
C_{A}(t, x)=\frac{1-x}{1-(1+t)x}.
$$
In particular, $C_{A}(-1, x)=1-x$ and thus $S_{A, 0}(n)=0$ for $n\geq 2$. As an immediate consequence of this property and Theorem \ref{prodrec}, we get that $S_{A, k}(n)=0$ for $n>k+2$. This shows a striking difference between the case of $A$-compositions and $A$-partitions. Indeed, it is conjectured in \cite{GU} that the sequence sign behaviour of the sequence $\left(\sum_{i=0}^{n}(-1)^{i}p_{A}(i,n)\right)_{n\in\N}$ is not eventually periodic.
}
\end{rem}
\bigskip

In Corollary \ref{oddset} we proved that if $A\subset 2\N+1$, then $(-1)^{n}S_{A,k}(n)\geq 0$ for each $k, n\in\N$. In the next theorem, we present a broad class of infinite sets $A\subset \N_{+}$ containing even integers such that $(-1)^{n}S_{A,k}(n)\geq 0$ for each $k, n\in\N$. More precisely, we have the following

\begin{thm}\label{oddeven1}
Let $E\subset 2\N_{+}$ and put $A=\N_{+}\setminus E$. Then, for each $k, n\in\N$ we have $(-1)^{n}S_{A,k}(n)\geq 0$.
\end{thm}
\begin{proof}
From Corollary \ref{prodrec} we know that to get the result it is enough to prove it for $k=0$. From our assumption we can write
$$
f_{A}(x)=\sum_{a\in A}x^{a}= \frac{x}{1-x}-f_{E}(x),
$$
where $f_{E}(x)=\sum_{e\in E}x^{e}$. Moreover, from the assumption on the set $E$ we have $f_{E}(x)=f_{E}(-x)$. Now, we have the following chain of equalities
\begin{align*}
\sum_{n=0}^{\infty}(-1)^{n}S_{A,0}(n)x^{n}=&\frac{1}{1+f_{A}(-x)}=\frac{1}{1-\frac{x}{1+x}-f_{E}(x)}\\
                                          =&\frac{1+x}{1-(1+x)f_{E}(x)}=(1+x)\cdot\frac{1}{1-f_{E'}(x)}\\
                                          =&(1+x)\sum_{n=0}^{\infty}c_{E'}(n)x^{n}=1+\sum_{n=1}^{\infty}(c_{E'}(n)+c_{E'}(n-1))x^{n},
\end{align*}
where $E'=E\cup (E+1)$ (note that from the assumption $E\subset 2\N_{+}$ we have $E\cap (E+1)=\emptyset$). We thus have the equality $(-1)^{n}S_{A,0}(n)=c_{E'}(n)+c_{E'}(n-1)$ for $n\in\N_{+}$. It is clear that this expression is non-negative.
\end{proof}

In the above theorem we constructed a broad class of {\it infinite} sets $A$ containing even (and odd) elements such that $(-1)^{n}S_{A,k}(n)\geq 0$. In the next result we present another method which allow us to construct {\it finite} sets $A$ containing even and odd elements with required property.

\begin{thm}\label{oddfinite}
Let $B\subset 2\N+1$ ba a finite set and $m=\#B$. Assume that the set of all sums of at most $m$ distinct elements of the set $B$ has exactly $2^{m}-1$ elements. Then, there is a finite set $A\subset\N$ such that $B\subset A, A\cap 2\N\neq \emptyset$ and for each $k, n\in\N$ we have $(-1)^{n}S_{A,k}(n)\geq 0$.
\end{thm}
\begin{proof}
From Corollary \ref{prodrec} we know that to get the result it is enough to prove it for $k=0$. Without loss of generality we can assume that $\#B\geq 2$. For $B=\{b_{1}, \ldots b_{m}\}$, where $m\in\N_{\geq 2}$ and $b_{1}, \ldots, b_{m}$ are odd, we consider the set $A$ containing all possible sums of the form
$\lambda_{1}+\ldots+\lambda_{j}$, where $\lambda_{1}<\lambda_{2}<\ldots<\lambda_{i}$ for $i\in\{1,\ldots, m\}$ and $\lambda_{j}\in B$. The elements of the set $A$ are coded by the exponent in the (polynomial) expansion $f_{A}(x)=\prod_{i=1}^{m}(1+x^{b_{i}})-1$. We then have the equality
\begin{align*}
F_{A}(-1,-x)&=\sum_{n=0}^{\infty}(-1)^{n}S_{A,0}(n)x^{n}=\frac{1}{1+f_{A}(-x)}\\
            &=\prod_{i=1}^{m}\frac{1}{1-x^{b_{i}}}.
\end{align*}
By comparing the coefficients near $x^{n}$ on both sides, we get the equality $(-1)^{n}S_{A,0}(n)=p_{B}(n)$, where $p_{B}(n)$ is the number of partitions with parts in the set $B$. It is clear that $p_{B}(n)\geq 0$ for $n\in\N$ and we get the result.
\end{proof}

In the above set of results we described several families of sets $A$ having the property $(-1)^{n}S_{A,k}(n)\geq 0$. However, as we will see, there are sets $A$ such that the sequence of signs of $S_{A,0}(n)$ is not eventually periodic.

\begin{thm}\label{negthm}
Let $p(x)\in\R[x]$ be a nonzero polynomial of the form
$$p(x)=c\prod_{i=0}^\rho(1-\omega_ix)^{\alpha_i},$$
where $\omega_i$ are pairwise different and $\alpha_i\in\N_{+}$ for $i=0, \ldots, \rho$. Assume that the number $\omega_0$ is non-real with $\op{Arg}(\omega_{0})\notin \Q\pi$ and such that its modulus is maximal, i.e., $|\omega_0|>|\omega_i|$ for $\omega_i\neq \omega_0,\overline{\omega_0}$. Then, if $1/p(x)=\sum_{k=0}^{\infty}a_{k}x^k$, then the sequence $(\op{sign}(a_k))_{k\in\N}$ is not eventually periodic.
\end{thm}
\begin{proof}
For convenience assume that $\omega_1=\overline{\omega_0}$. Set $\omega=\omega_0$. We have
$$
a_k=P(k)\omega^k+Q(k)\overline{\omega}^k+\sum_{i>1}P_i(k)\omega_i^k,
$$
where $P_i, P, Q$ are polynomials with coefficients in $\C$. The conjugation on $\C$ extends to the automorphism of the ring $\C[x]$ in the following way:
$$
\C[x]\ni f=\sum_{i=0}^{k}a_ix^i\mapsto \overline{f}:=\sum_{i=0}^{k}\overline{a_i}x^i\in\C[x].
$$
\textbf{Claim: }$Q=\overline{P}$\\
We will postpone the proof of the claim to the end. Let us write $P(z)=bz^n+R(z)$, where $\op{deg}R<n.$ We have the following chain of equalites
\begin{align*}
\frac{a_k}{|b\omega^k|k^n}&=\frac{R(k)}{|b|k^n}\zeta^k+\frac{\overline{R}(k)}{|b|k^n}\overline{\zeta}^k+\frac{b}{|b|}\zeta^k+\frac{\overline{b}}{|b|}\overline{\zeta}^k+\sum_{i>1} \frac{P_i(k)}{|b|k^n}\left(\frac{\omega_i}{|\omega|}\right)^k\\
&=2\text{Re}(c\zeta^k)+o(1),
\end{align*}
where $\zeta=\omega/|\omega|$ and $c=b/|b|.$ Finally, we get that
$$
\op{sign}(a_k)=\op{sign}\left(\frac{a_k}{|bz^k|k^n}\right)=\op{sign}(2\op{Re}(c\zeta^n)+o(1)).
$$
Set $c=e^{2\pi i \gamma}$ and $\zeta=e^{2\pi i \delta}$. The assumption $\op{Arg}(w)\notin \Q\pi$ says that $\delta$ is irrational. Then $2\op{Re}(c\zeta^n)$ depends only on the fractional part $\{\gamma+n\delta\}$. Also looking at the sign of $2\text{Re}(c\zeta^n)+o(1)$ we can ignore error as long as $2\op{Re}(c\zeta^n)$ is big enough, i.e., the corresponding fractional part is far from $\frac{1}{2}\Z.$  Now, if the sequence $(\op{sign}(a_k))_{k\in\N}$ is eventually periodic then we can find $d$ such that $\op{sign}(a_{kd})$ is constant for sufficiently large $k$. This means that $\{\gamma+k(d\delta)\}$ always omits the interval $(\epsilon, 1/2-\epsilon)$ or the interval $(1/2+\epsilon,1-\epsilon)$. However, the number $\delta$ (and so $d\delta$) is irrational, and hence the sequence $(\gamma+k(d\delta))_{k\in\N}$ is equidistributed. This is a contradiction, and we are only left with proving the claim.\\
\textbf{Proof of the claim:}
Let
$$
S(x)=c(x-b_{1})^{s_1}\ldots (x-b_{m})^{s_m}(x^2+q_{1}x+r_{1})^{t_1}\ldots(x^2+q_{l}x+r_{l})^{t_n}
$$
be a decomposition of the polynomial $S$ into irreducible factors in $\R[x].$ Then
$$
1/p(x)=\sum_{i=1}^m\sum_{s=1}^{s_i}\frac{A_{i,s}}{(x-b_i)^s}+\sum_{j=1}^l\sum_{t=1}^{t_j}\frac{B_{j,t}x+C_{j,t}}{(x^2+q_jx+r_j)^t},
$$
where $A_{i,s},B_{j,t},C_{j,t}\in\R$. Notice that it is enough to prove the claim for $\frac{1}{(x^2+qx+r)^t}.$ Write $x^2+qx+r=(x-z)(x-\overline{z})$ and observe that
$$
\frac{1}{(x^2+qx+r)^{t}}=\frac{f_1(x)}{(x-z)^t}+\frac{f_2(x)}{(x-\overline{z})^t},
$$
for $f_{1}, f_{2}\in\C[x]$ satisfying $\op{deg}f_1,\op{deg}f_2<t.$ Now, if we prove that $f_2=\overline{f_1}$, then we are done. However, we have
$$
1=f_1(x)(x-z)^t+f_2(x)(x-\overline{z})^t=\overline{f_1}(x)(x-\overline{z})^t+\overline{f_2}(x)(x-z)^t,
$$
and hence $(x-z)^t\mid f_2(x)-\overline{f_1}(x)$. Because the degree of $f_2-\overline{f_1}$ is $<t$, we get that $f_2=\overline{f_1}$. This ends the proof of the claim and of the whole theorem.
\end{proof}
\begin{cor}
If $A=\{2, 3\}$ or $A=\{1,4\}$ then the sequence $(\op{sign}(S_{A,0}(n)))_{n\in\N}$ is not eventually periodic.
\end{cor}
\begin{proof}
We just need to check that assumptions of previous theorems hold in both cases. Let us take first $A=\{2,3\}$ which gives $p(x)=1+x^2+x^3.$ Let $t$ be the unique real root  of $p$ and $z$ be its complex root. We have $x\approx -1.46$ and $z\approx 0.23-0.79i$, and thus indeed $|z|<|x|.$ The last thing to check is that $\zeta=z/|z|$ is  not a root of unity. We have that $[\Q(\zeta):\Q]$ is at most 12 so it is enough to check that $\zeta^{12}\neq 1$ but $\zeta^{12}\approx -0.95-0.28i$ so all the conditions of Theorem \ref{negthm} are met. Similarly for the second case.
\end{proof}

\section{Some problems and a question}\label{sec4}

In Theorem \ref{oddfinite} we constructed a broad class of finite sets $A$ containing even (and odd) elements such that $(-1)^{n}S_{A,k}(n)\geq0$. However, one can go further and state the following general problem.

\begin{prob}
Let $A\subset \N_{+}$ and suppose that there exist $k\in\N$ such that for infinitely many $n$ the inequality $(-1)^{n}S_{A,k}(n)<0$ holds. Construct an optimal set $B\supset A$, i.e.,  with the minimal possible values of $\#(B\setminus A)$ and $\op{min}(B)$, such that $(-1)^nS_{B,k}(n)\geq0$ for every $k$.
\end{prob}
From Proposition \ref{12m} we know that if $A=\{1,\ldots, m\}$ and $m$ is even, the optimal set for $A$ is $B=A\cup\{m+1\}$. However, in general, it is not clear whether, for a given set $A$, an optimal set $B$ is uniquely determined.




\begin{ques}
Let $N\in\N_{+}$. What can be said about the quantity
$$
F(N):=\#\{A\subset\{1,\ldots, N\}:\;\forall k\in\N\;\forall! n\in\N:\;(-1)^{n}S_{A,k}(n)\geq 0\}?
$$
\end{ques}

From Corollary \ref{oddset} we know that $F(n)\geq 2^{\lfloor\frac{n+1}{2}\rfloor}$.

\begin{prob}
For which disjoint sets $A, B\subset\N_{+}$ the conditions $(-1)^{n}S_{A,0}(n)\geq 0, (-1)^{n}S_{B,0}(n)\geq 0$ implies $(-1)^{n}S_{A\cup B,0}(n)\geq 0$?
\end{prob}

One can formulate the following simpler problem.

\begin{prob}
Let $B\subset 2\N+1$. Does there exist an even number $b$ such that for $A=B\cup \{b\}$ we have $(-1)^{n}S_{A,0}(n)\geq 0$?
\end{prob}

It seems that if $m>3$ is even and we take $B=\{\frac{m^{i}-1}{m-1}:\;i\in\N_{+}\}$ then one can take $b=m$. In general let us note that if $A\cap B=\emptyset$, then $f_{A\cup B}(x)=f_{A}(x)+f_{B}(x)$ and thus, using the equalities
$$
C_{A}(t,x)=\frac{1}{1-tf_{A}(x)},\quad C_{B}(t,x)=\frac{1}{1-tf_{B}(x)},\quad C_{A\cup B}(t,x)=\frac{1}{1-tf_{A\cup B}(x)},
$$
we can eliminate $f_{A}, f_{B}, f_{A\cup B}$ and get the functional relation
$$
\frac{1}{C_{A}(t,x)}+\frac{1}{C_{B}(t,x)}-\frac{1}{C_{A\cup B}(t,x)}=1,
$$
and natural question arises whether is it possible to use it to deduce any sufficient conditions on $A, B$ which will guarantee that $(-1)^{n}S_{A\cup B,0}(n)\geq 0$ for all $n\in\N$.

\vskip 1cm

\noindent Filip Gawron, Jagiellonian University, Faculty of Mathematics and Computer Science, Institute of Mathematics, {\L}ojasiewicza 6, 30-348 Krak\'ow, Poland; email: filipux.gawron@doctoral.uj.edu.pl

\bigskip

\noindent Maciej Ulas, Jagiellonian University, Faculty of Mathematics and Computer Science, Institute of
Mathematics, {\L}ojasiewicza 6, 30-348 Krak\'ow, Poland; email:
maciej.ulas@uj.edu.pl

\end{document}